\documentclass[11pt]{amsart}

\usepackage{amsmath, amssymb} 

\newcommand{\limi}[1]{\lim_{{#1} \to \infty}}

\newcommand{\af}{\alpha}
\newcommand{\bt}{\beta}

\newcommand{\dt}{\delta}
\newcommand{\ep}{\varepsilon}

\newcommand{\et}{\eta}

\newcommand{\ld}{\lambda}

\newcommand{\rh}{\rho}
\newcommand{\om}{\omega}
\newcommand{\ta}{\tau}

\newcommand{\Ph}{\Phi}

\newcommand{\Q}{{\mathbb{Q}}}

\newcommand{\R}{{\mathbb{R}}}
\newcommand{\C}{{\mathbb{C}}}
\newcommand{\N}{{\mathbb{N}}}

\DeclareMathOperator{\spec}{sp}
\DeclareMathOperator{\spn}{span}

\newcommand{\BH}{L (H)}

\newcommand{\andeqn}{\,\,\,\,\,\, {\mbox{and}} \,\,\,\,\,\,}

\newcommand{\wolog}{without loss of generality}

\newcommand{\ca}{C*-algebra}

\newcommand{\pj}{projection}

\newcommand{\pisca}{purely infinite simple \ca}
\newcommand{\kfalg}{separable nuclear unital \pisca}

\newcommand{\CH}{Continuum Hypothesis}

\newcommand{\I}{\infty}

\newcommand{\bbD}{{\mathbb{D}}}
\newcommand{\bbF}{{\mathbb{F}}}
\newcommand{\bbE}{{\mathbb{E}}}

\newcommand{\bbN}{{\mathbb{N}}}
\newcommand{\bbC}{\mathbb{C}}

\newcommand{\rs}{\restriction}

\DeclareMathOperator{\SPAN}{span}

\newcommand{\cC}{{C}}
\newcommand{\cF}{\mathcal{F}}

\newcommand{\cK}{{K}}

\newcommand{\calD}{\mathcal{D}}

\newcommand{\e}{\varepsilon}
\newtheorem{thm}{Theorem}[section]
\newtheorem{theorem}{Theorem}
\newtheorem{coro}[thm]{Corollary}
\newtheorem{corollary}[theorem]{Corollary}

\newtheorem{question}[thm]{Question}

\newtheorem{lemma}[thm]{Lemma}
\newtheorem{lem}[thm]{Lemma}
\newtheorem{prop}[thm]{Proposition}

\theoremstyle{definition}

\newtheorem{notation}[thm]{Notation}
\newtheorem{definition}[thm]{Definition}

\newcounter{my_enumerate_counter}
\newcommand{\pushcounter}{\setcounter{my_enumerate_counter}{\value{enumi}}}
\newcommand{\popcounter}{\setcounter{enumi}{\value{my_enumerate_counter}}}

\def\rs{\restriction}

\newcommand{\lbl}{\label}

\newcommand{\bfa}{{\mathbf{a}}}
\newcommand{\bfb}{{\mathbf{b}}}
\newcommand{\bfc}{{\mathbf{c}}}

\newcommand{\cP}{{\mathcal{P}}}

\newcommand{\cU}{{\mathcal{U}}}
\newcommand{\cV}{\mathcal{V}}
\newcommand{\cW}{\mathcal{W}}

\DeclareMathOperator{\Ann}{Ann}

\title[The commutant of $\BH$ in its ultrapower]
{The commutant of $\BH$ in its ultrapower may or may not be trivial}

\author{Ilijas Farah}

\address{Department of Mathematics and Statistics\\
York University\\
4700 Keele Street\\
North York, Ontario\\
Canada, M3J 1P3\\
and 
Matematicki Institut \\
Kneza Mihaila 34 \\
Belgrade \\
Serbia}

\urladdr{http://www.math.yorku.ca/$\sim$ifarah}

\email{ifarah@mathstat.yorku.ca}

\thanks{Ilijas Farah and Juris Stepr\={a}ns
 were partially supported by NSERC.
 N.~Christopher Phillips was partially supported
 by NSF grant DMS-0701076,
 by the Fields Institute for Research in Mathematical Sciences,
 Toronto, Canada,
 and by an Elliott Distinguished Visitorship at the Fields Institute}

\author{N.~Christopher Phillips}

\address{Department of Mathematics\\
University of Oregon\\ Eugene, Oregon 97403-1222, U.S.A.}

\urladdr{http://www.uoregon.edu/$\sim$ncp/}

\email{ncp@uoregon.edu}

\author{Juris Stepr\={a}ns}

\address{Fields Institute\\ 222 College Street\\ Toronto, ON\\ Canada}
\email{steprans@fields.utoronto.ca}


\date{\today}

\begin{document}

\begin{abstract}
Kirchberg asked in 2004 whether the commutant of $\BH$
in its (norm) ultrapower is trivial.
Assuming the Continuum Hypothesis,
we prove that the answer depends on the choice of the ultrafilter.
\end{abstract}

\maketitle

Let $H$ be a separable infinite dimensional complex Hilbert space,
fixed throughout. The purpose of this paper is to prove that,
assuming the Continuum Hypothesis, the commutant of $\BH$ (the
algebra of bounded linear operators on $H$) in its ultrapower depends
on the choice of the ultrafilter. This provides a somewhat
surprising---and somewhat incomplete---answer to Question~2.22
of~\cite{Kirc:Central}.

We follow the convention that $0\in \bbN$.
Let $A$ be a \ca,
and let $\cV$ be a nonprincipal ultrafilter on $\bbN$
(equivalently, a point in $\bt \N \setminus \N$).
(In the operator algebra literature,
nonprincipal ultrafilters are usually denoted~$\om$;
since in the set theory literature
$\om$ is reserved for the least infinite ordinal we
suppress using this symbol altogether.)
We denote by $\ell^{\infty} (A)$ the \ca\ of all bounded functions
from $\N$ to $A$,
and we denote by $A^{\cV}$ the
ultrapower of~$A$ associated with~$\cV$ (\cite{GeHa}, \cite{BYBHU}).
More explicitly, $A^{\cV}$ is the quotient of $\ell^{\infty} (A)$
by the (two-sided, norm-closed, selfadjoint) ideal
\[
c_{\cV} (A) =
\big\{ (a_n) \in \ell^{\infty} (A)
  \colon {\mbox{$\lim_{n \to \cV} \| a_n \| = 0$}} \big\}.
\]
(This algebra is denoted $A_{\cV}$
in \cite{Kirc:Central} and~\cite{KP1}.)
We identify $A$ with the subalgebra of $A^{\cV}$
consisting of all
\[
(a, a, a, \ldots ) + c_{\cV} (A)
\]
for $a \in A$, that is, the image in $A^{\cV}$ of all constant sequences
in $\ell^{\infty} (A)$.
If $A \subseteq B,$
we denote by $A' \cap B$ the relative commutant of~$A$ in~$B,$
that is
\[
A' \cap B
 = \big\{ b \in B \colon {\mbox{$a b = b a$ for all $a \in A$}} \big\}.
\]
(By convention, $A'$ is used
only when $B$ is the algebra of bounded operators on a Hilbert space.)
Following Kirchberg~\cite{Kirc:Central},
we define the invariant $F_{\cV} (A)$ of $A$ to be the
quotient of
$A'\cap A^{\cV}$ by the annihilator
\[
\Ann \big( A, A^{\cV} \big)
 = \{b \in A^{\cV} \colon {\mbox{$b a = 0$ for all $a \in A$}} \}.
\]
If $A$ is unital then the annihilator is trivial and
\[
F_{\cV} (A) = A' \cap A^{\cV}.
\]

The algebra $F_{\cV} (A)$ for unital~$A$,
often denoted $A_{\cV}$ in the literature (see~\cite{Iz}),
has played an important role in the study of \ca s,
particularly in the classification of \kfalg s
and group actions on them.
(See \cite{Kr}, \cite{KP1}, \cite{Ph}, and~\cite{Iz}.)
The most important result, proved for example in~\cite{KP1},
is that if $A$ is a \kfalg,
then $F_{\cV} (A)$ is again a \pisca.
See~\cite{GeHa} for further discussion of the uses of algebras
of this type,
and~\cite{Kirc:Central} for more recent applications
of $F_{\cV} (A)$.

The use of $F_{\cV} (A)$ for \pisca s parallels an older use
of its tracial analog for factors of type II$_1$.
We briefly describe this use to give context,
but it is not needed in the sequel.
In place of operator norm convergence,
one uses convergence in the $L^2$~norm derived from the trace
in the definition of $c_{\cV} (A)$.
For example, if $M$ is the hyperfinite factor of type II$_1$,
then the tracial analog of $F_{\cV} (M)$
is again a factor of type II$_1$.
(See Lemma XIV.4.5 and Theorems XIV.4.6 and XIV.4.18 of~\cite{Tk3}.
This is also in~\cite{Cn2},
and in~\cite{FaHaSh:Model} it was shown that this is a consequence of
the Fundamental Theorem for Ultraproducts from logic.)
As one example, this fact is used to prove that outer actions
of suitable groups on the hyperfinite factor of type II$_1$
have the Rokhlin property,
in turn a key step in the classification of such actions.
See~\cite{Jn} for the case of finite groups,
and~\cite{Oc} for the case of countable amenable groups.

Assuming the \CH, Ge and Hadwin
proved (\cite[Corollary~3.4]{GeHa})
that if $A$ is separable,
then the isomorphism class of $F_{\cV} (A)$ is
independent of the choice of the
nonprincipal ultrafilter~$\cV$.
But $\BH$ is not, of course, separable in the norm topology.
(The easiest way to see this is to note that projections
onto the closed subspaces spanned by all subsets of a
fixed orthonormal basis form an uncountable discrete set.)

Let $\cK (H)$ denote the ideal of
compact operators in $\BH$.
For the Calkin algebra $\cC (H) = \BH / \cK (H)$,
Kirchberg proved (\cite[Corollary~2.21]{Kirc:Central})
that $F_{\cV} (\cC (H)) = \bbC$.
This also implies  $F_{\cV} (\BH) \subseteq \bbC + \cK (H)^{\cV}$.
In \cite[Question~2.22]{Kirc:Central},
Kirchberg asked whether $F_{\cV} (\BH ) = \bbC$.
The answer to this question is somewhat surprising.

\begin{theorem}\lbl{T1}
There is a nonprincipal ultrafilter $\cV$
such that $F_{\cV} (\BH ) \neq \bbC$.
\end{theorem}

\begin{proof}
This follows from
Theorem~\ref{T.Flat.1} and Theorem~\ref{T.Comm.1}.
\end{proof}

\begin{theorem}\lbl{T2}
Assume the \CH.
Then there is a nonprincipal ultrafilter $\cV$
such that $F_{\cV} (\BH) = \bbC$.
\end{theorem}

\begin{proof}
This follows from Corollary~\ref{C.S.1}, where we prove $F_{\cV}
(\BH) = \bbC$ for a selective ultrafilter $\cV$,
and from the fact that the
Continuum Hypothesis implies selective ultrafilters exist
(Proposition~\ref{P.selective}).
\end{proof}

We record a curious consequence of Theorem~\ref{T1}
and Theorem~\ref{T2}.

\begin{corollary} \lbl{Co.1}
Assume the Continuum Hypothesis.
Then the isomorphism type of
$F_{\cV} (\BH)$ depends on the choice of the ultrafilter~$\cV$.
\qed
\end{corollary}

It should be noted that in the absence of the \CH,
$F_{\cV} (A)$ depends
on the choice of $\cV$
for every infinite-dimensional separable C*-algebra $A$
(\cite{Fa:Relative} for real rank zero
and~\cite{FaHaSh:Model} in the general case).
We don't know
whether some axiom beyond ZFC  is needed for the conclusion of
Corollary~\ref{Co.1} but it should be noted that an assumption much
weaker than \CH\ suffices.
As noted above, the \CH\  implies that the
isomorphism type of $F_{\cV} (A)$ does not depend on $A$ for a
separable $A$. The reason is that, under the \CH, for any two
nonprincipal ultrafilters $\cV$ and $\cW$ on $\bbN$ there is an
isomorphism between $A^{\cV}$ and $A^{\cW}$ that sends $A$ to itself
(\cite[Theorem~3.1]{GeHa}). The latter fact is an immediate
consequence of the fact, provable in Zermelo-Fraenkel set theory with
the Axiom of Choice (ZFC), that the unit ball of $A^{\cV}$ is
{\emph{countably saturated}} in the logic of metric structures.
(See~\cite[\S 4.4]{FaHaSh:Model2}.)
Since the \CH\  implies that the
ultrapowers are of size $\aleph_1$, a back-and-forth construction
easily gives an isomorphism between $A^{\cV}$ and $A^{\cW}$ that
sends the copy of $A$ in one ultrapower to the copy of $A$ in the
other ultrapower.
(See~\cite{FaHaSh:Model2} for definitions.)

One curiosity deserves a mention here. By the countable saturatedness
of ultrapowers, the Continuum Hypothesis implies that $\BH^{\cU}$ and
$\BH^{\cV}$ are isomorphic for any two nonprincipal ultrafilters
$\cU$ and $\cV$ on $\bbN$. However, if $\cU$ is selective and $\cV$
is flat, then by Corollary~\ref{Co.1} we have  $F_{\cV}(\BH)\not\cong
F_{\cU}(\BH)$ and therefore no isomorphism between $\BH^{\cU}$ and
$\BH^{\cV}$ sends $\BH$ to $\BH$.

\subsection*{Acknowledgments}
We would like to thank Eberhard Kirchberg for useful conversations. The
results of this paper were obtained during the Fall 2007 Fields
Institute semester on operator algebras and the May 2008 Canadian
Operator Symposium in Toronto.

\section{Selective ultrafilters}

\begin{definition}\label{D:Central}
Let $M$ be a C*-algebra or a von Neumann algebra.
We let $[a,b]$ denote the additive commutator, $ab-ba$.
Following~\cite{Kirc:Central}
we say that a norm-bounded sequence $(a_n)_{n \in \N}$ in $M$
is {\emph{central}} if for
every $b \in M$ we have $\lim_{n \to \infty} \| [a_n, b] \|=0$.
A sequence $(a_n)_{n \in \N}$ is {\emph{trivial}}
if
\[
\lim_{n \to \infty} \inf_{\ld \in \bbC} \| a_n - \ld \| = 0.
\]
(Note that such sequences are always central.)
If $\cV$ is an
ultrafilter on $\bbN$
then a norm-bounded sequence $(a_n)_{n \in \N}$ is a
{\emph{$\cV$-central sequence}}
if for every $b \in M$ we have $\lim_{n \to \cV}
\| [a_n, b] \| = 0$.
A $\cV$-central sequence $(a_n)_{n \in \N}$ is {\emph{trivial}}
if $\lim_{n \to \cV} \inf_{\ld \in \bbC} \| a_n - \ld \| = 0$.
\end{definition}

Note that we have defined norm central sequences
and norm trivial central sequences,
even in the case of a von Neumann algebra.
These are not the same as the central sequences usually
considered for a type II$_1$ factor.

\begin{notation}\label{N:Bracket}
If $X$ is a set then $[X]^2$ denotes the set of all two-element
subsets of $X$ and $[X]^{\infty}$ denotes the set of all infinite
subsets of $X$. We consider the space $[\bbN]^{\infty}$ with the
Polish (that is, separable and completely metrizable) topology
inherited from the Cantor set (identified with the
power set of $\bbN$).
\end{notation}

Recall that a subset of a
Polish space is {\emph{analytic}} if it is a continuous image of a
Borel subset of a Polish space. The following result is well-known,
but we sketch a proof of the easy implications for convenience of the
reader.

\begin{thm}[A.~R.\  D.\  Mathias] \label{T.S.1}
The following are equivalent for an ultrafilter $\cV$ on $\bbN$.
\begin{enumerate}
\item \label{T.S.1.0}
If $X_0 \supseteq X_1 \supseteq X_2 \supseteq \dots$
are in
$\cV$ then there exists $X\in \cV$
such $X\setminus \{0,1,\dots, n\}\subseteq X_n$ for every $n\in X$.
\item\label{T.S.1.1}
For every $g\colon \bbN\to \bbN$ there is
$A \in \cV$ such that $g \rs A$
is either constant or injective.
\item\label{T.S.1.2}
For every $E\subseteq [\bbN]^2$ there is $A\in \cV$
such that either
$[A]^2\subseteq E$ or $[A]^2\cap E=\varnothing$.
\item\label{T.S.1.3}
For every analytic $\bbE\subseteq [\bbN]^{\infty}$
there is $A\in \cV$
such that either
$[A]^{\infty} \subseteq \bbE$ or $[A]^{\infty} \cap \bbE=\varnothing$.
\end{enumerate}
\end{thm}

\begin{proof}
The fact that \eqref{T.S.1.0} implies \eqref{T.S.1.3} is Mathias's
theorem \cite[Theorem~0.13]{Mat:Happy}.
We include proofs of
the easy implications for the convenience of the reader.

We consider \eqref{T.S.1.3} implies~\eqref{T.S.1.2}.
Let $E \subseteq [\bbN]^2$.
Let $\bbE \subseteq [\bbN]^{\infty}$
be the set of all $X\in [\bbN]^{\infty}$ such that
the pair consisting of the least two elements of $X$ is in $E$.
Then $\bbE$ is analytic (and even open).

For \eqref{T.S.1.2} implies \eqref{T.S.1.1}, given $g$ let
$\{m,n\}\in E$ if and only if $g(m)\neq g(n)$.

For \eqref{T.S.1.1} implies \eqref{T.S.1.0}, we may assume $X_0=\bbN$
and $\bigcap_{n = 0}^{\infty} X_n = \varnothing$.
Define $g (k) = n$ if $k \in X_n \setminus X_{n + 1}$.
Since $\cV$ contains no set on which $g$ is constant,
\eqref{T.S.1.1} implies that
there is $Y \in \cV$ such that $g$ is injective on $Y$.

Recursively find $0 = m_0 < m_1 < \cdots$ so that for all $k$ and
all $j\geq m_{k+1}$ in $Y$,
we have $g (j) \geq m_k$.
Let $g_1\colon \bbN\to \bbN$ be
defined by $g_1(j)=k$ if $m_k \leq j < m_{k + 1}$. Let $Z\subseteq Y$
be a set in $\cV$ on which $g_1$ is injective.

Assume for the moment that
$X = Z \cap \bigcup_{k = 0}^{\infty} [m_{2 k}, \, m_{2 k + 1})$
is in $\cV$.
For $n \in X$ let $k$ be the unique integer
such that $m_{2k} \leq n < m_{2k+1}$.
Then $X \setminus \{ 0, 1, \ldots, n \}$
is disjoint from $\{ 0, 1, \dots, m_{2k+2} - 1\}$.
Therefore for $m \in X \setminus \{0, 1, \ldots, n\}$
we have $g (m) \geq m_{2 k + 1} > n$.
Hence $m \in X_n$ as required.

If $Z\cap \bigcup_{k=0}^{\infty} [m_{2k}, \, m_{2k+1})\notin \cV$,
we
instead set $X=Z\cap \bigcup_{k=0}^{\infty} [m_{2k+1}, \, m_{2k+2})$.
This set is in $\cV$ and
an argument analogous to the above shows that
$X \setminus \{0, 1, \ldots, n \} \subseteq X_n$ for all $n\in X$.
\end{proof}

An ultrafilter satisfying the conditions of Theorem~\ref{T.S.1} is
said to be {\emph{selective}} (or {\emph{Ramsey}}, or even
{\emph{happy}}).
More information on these remarkable objects can be found in
\cite{Mat:Happy} or~\cite{Fa:Semiselective}.
(For example, see Section~5
and Theorem~4.9 of~\cite{Fa:Semiselective}.)

It is well-known that the existence of selective ultrafilters can be
deduced from the Continuum Hypothesis
(\cite[Proposition~0.11]{Mat:Happy}).
We sketch a proof of this
fact for the convenience of the reader.

\begin{prop}\label{P.selective} Assume the Continuum Hypothesis.
Then there exists a nonprincipal  selective ultrafilter.
\end{prop}

\begin{proof}
Let $\omega_1$ be the first uncountable ordinal.
Use the Continuum Hypothesis to enumerate all functions
$g\colon\bbN\to \bbN$ as $g_{\gamma}$, for $\gamma<\omega_1$.
We claim that there are
infinite sets $A_{\gamma}$ for $\gamma<\omega_1$ such that the
following holds for $\gamma<\omega_1$:
\begin{enumerate}
\item $g_{\gamma}$ is either constant or injective on $A_{\gamma}$.
\item $A_{\gamma}\setminus A_{\delta}$ is finite if $\delta<\gamma$.
\end{enumerate}

We prove the claim by transfinite induction.
If $g_0$ is constant on
an infinite subset of $\N$, let $A_0$ be this set.
Otherwise, the image of $\N$ under $g_0$ is infinite
and we can find an infinite subset $A_0 \subseteq \N$
on which $g_0$ is injective.

Assume $A_{\gamma}$ has been chosen.
Use the argument above, with $A_{\gamma}$ in place of $\N$
and $g_{\gamma + 1}$ in place of $g_0$,
to choose an infinite subset $A_{\gamma + 1} \subseteq A_{\gamma}$
on which $g_{\gamma + 1}$ is either constant or injective.

Now assume $\gamma$ is a limit ordinal and the sets $A_{\delta}$ have
been chosen for $\delta<\gamma$.
Re-enumerate the sets $A_{\delta}$ for $\delta<\gamma$
as $A_j'$ for $j \in \bbN$.
For $j\in \bbN$ let
\[
\eta = \max \big( \big\{ \delta < \gamma \colon
    {\mbox{$A_{\delta}=A_k'$ for some $k < j$}} \big\} \big).
\]
Then the set  $A_{\eta} \setminus \bigcap_{k < j} A_k'$ is finite,
and in particular $\bigcap_{k<j}A_k'$ is infinite for every $j$.
We can therefore choose a sequence $n_0 < n_1 < n_2 < \cdots$
so that $n_j \in \bigcap_{k < j} A_k'$ for all $j \in \bbN$.
Set $B = \{ n_j \colon j \in \bbN \}$.
Then $B$ is infinite and
$B \setminus A_{\delta}$ is finite for all $\delta < \gamma$.
Now $A_{\gamma} \subseteq B$ is chosen as above.

The family $\{A_{\gamma} \}_{\gamma\in\omega_1}$
has the finite intersection property---indeed,
for any finite set $F\subseteq \omega_1$,
if $\gamma = \max(F)$ then there is a finite set
$a\subseteq A_{\gamma}$ such that
$A_{\gamma} \setminus a \subseteq \bigcap_{\delta\in F} A_{\delta}$.
Hence this family can be extended to an ultrafilter,
and this ultrafilter is selective by construction.
\end{proof}

For  $\bfa \in A^{\cV}$
a sequence $(a_n)_{n \in \N}$ in $\ell^{\infty}(A)$
such that
$\bfa = (a_n)_{n \in \N} + c_{\cV} (A)$
is called a {\emph{representing sequence}}.
A von Neumann algebra $M$ is {\emph{separably acting}}
if it is isomorphic to
a weak operator closed selfadjoint subalgebra of the bounded operators
on a separable Hilbert space.
Thus $\BH,$
for our fixed infinite dimensional separable Hilbert space~$H,$
is separably acting.

\begin{thm}\label{T.S.2.0}
Assume $M$ is a separably acting von Neumann
algebra and
$\cV$ is a selective ultrafilter.
Then for
$\bfa \in M^{\cV}$ the following are equivalent.
\begin{enumerate}
\item\label{T.S.2.0:1}
$\bfa \in M' \cap M^{\cV}$.
\item\label{T.S.2.0:2}
$\bfa$ has a representing sequence that is a central sequence.
\end{enumerate}
\end{thm}

\begin{proof}
We need only prove that (\ref{T.S.2.0:1}) implies~(\ref{T.S.2.0:2}).
Fix $\bfa \in M^{\cV}$ and a representing sequence $(a_n)_{n \in \bbN}$.

The closed unit ball ${\mathbf{B}}$ of $M$
with the weak operator topology is compact metric,
and is therefore a Polish space.
See page~35 of~\cite{Dx}.
For the convenience of the reader, we sketch the proof here.

We may assume $M \subseteq \BH$
for our fixed infinite dimensional separable Hilbert space~$H.$
Let $D \subseteq H$ be a countable norm dense subset,
and let $C \subseteq \C$ be a countable dense subset.
Observe that the weak operator topology on ${\mathbf{B}}$
has a base consisting of sets of the form
\[
\bigcap_{j = 1}^n \big\{ a \in {\mathbf{B}}
  \colon  | \langle a \xi_j, \et_j \rangle - \ld_j | < 1 / k_j \big\}
\]
with $n \in \N$ and
\[
\xi_1, \xi_2, \ldots, \xi_n, \et_1, \et_2, \ldots, \et_n \in D,
\,\,
\ld_1, \ld_2, \ldots, \ld_n \in C,
\,\, {\text{and}} \,\,
k_1, k_2, \ldots, k_n  \in \N.
\]
So ${\mathbf{B}}$
with the weak operator topology is second countable.
To see that it is a Polish space use the Banach-Alaoglu theorem
to conclude that it is, in fact, compact.

The power set $\cP (\bbN)$
with the Cantor set topology is compact and metrizable and,
hence, also a Polish space.

We will prove that the set
\[
\bbE = \big\{ X \in [ \bbN ]^{\infty}
    \colon ( \exists j \in \bbN )
	   (\exists b \in {\mathbf{B}} ) (\forall n \in X )
	\| [a_n, b] \| > 1/j \big\}
\]
is analytic. 	

For each $\e > 0$ define
$\Phi_{\e} \colon {\mathbf{B}} \times \cP (\bbN) \to \cP (\bbN)$ by
\[
\Phi_{\e} (b, X)
 = \{ n \in X\colon \| [a_n, b] \| > \e \}
\]
for $b \in {\mathbf{B}}$ and $X \subseteq \bbN$.
We claim that
$\Phi_{\e}$ is Borel measurable.
First, since multiplication is separately weak operator continuous
and norm closed balls in $M$ are also weak operator
closed, for each fixed $n$ the set
\[
J_n
 = \{b \in {\mathbf{B}} \times \cP (\bbN) \colon
             \| [a_n, b] \| > \e \}
\]
is open in ${\mathbf{B}}$.
Now let $U \subseteq \cP (\bbN)$ be a basic open set,
that is, there are $N \in \bbN$ and
a subset $S \subseteq \{ 0, 1, \ldots, N \}$ such that $U$
consists of all $X \subseteq \N$
with $X \cap \{ 0, 1, \ldots, N \} = S$.
Then
\begin{multline*}
\Phi_{\e}^{-1} ( U ) =
\big\{ (b, X) \in {\mathbf{B}} \times \cP (\bbN)
     \colon S = X\cap \{ 0, 1, \ldots, N \} \big\} \\
      \cap \left(\bigcap_{n\in S} J_n \right) \cap \left(
     \bigcap_{n \in  \{ 0, 1, \ldots, N \} \setminus S}
             {\mathbf{B}} \setminus J_n \right),
\end{multline*}
establishing that  $\Phi_{\e}^{-1} ( U )$ is Borel.
This proves that $\Phi_{\e}$ is Borel measurable.

Define
$Z = \coprod_{j = 1}^{\infty} [{\mathbf{B}} \times \cP (\bbN)]$,
and define $\Ph \colon Z \to \cP (\bbN)$
by letting $\Ph = \Ph_{1/j}$ on the $j$ component of the disjoint union.
Since $M$ acts on a separable Hilbert space, $Z$ is a Polish
space. Clearly $\Ph$ is Borel.
The subset $[\bbN]^{\infty} \subseteq \cP (\bbN)$ is Borel,
since its complement is countable.
Therefore $\Ph^{-1} ( [\bbN]^{\infty} )$ is Borel and the set
\[
\bbE
 = \Ph \big( \Ph^{-1} ( [\bbN]^{\infty} ) \big)
 = [\bbN]^{\infty}
      \cap \bigcup_{j = 1}^{\infty} \Phi_{1/j}
           [{\mathbf{B}} \times \cP (\bbN)]
\]
is analytic.
Since $\cV$ is selective,
there is $X \in \cV$ such that $[X]^{\infty} \subseteq \bbE$ or
$[X]^{\infty} \cap \bbE=\varnothing$.

Let us first assume the second possibility applies.
Then for each $b\in {\mathbf{B}}$ and $j \in \bbN$ with $j > 0$,
the set $\{n \in X \colon \| [a_n, b] \| > 1/j \}$ is not in $\bbE$,
and therefore must be finite.
Let $a_n'=a_n$ if $n\in X$ and $a_n'=1$ if $n\notin X$.
Then $\{n \in \bbN \colon \| [a_n', b] \| > \ep \}$ is finite for all
$b \in {\mathbf{B}}$ and $\ep > 0$, so $(a'_n)_{n \in \bbN}$ is a
central sequence. It represents $\bfa$ since $X \in \cV$.

Now assume there is $X\in \cV$ such that $[X]^{\infty} \subseteq \bbE$.
In particular, there are $b \in {\mathbf{B}}$ and $j \in \bbN$
such that $\|[a_n,b]\|\geq 1/j$ for all $n \in X$;
therefore $\bfa \notin M' \cap M^{\cV}$.
\end{proof}

\begin{coro}\label{C.S.2}
Assume $M$ is a separably acting von Neumann algebra.
If $\cV$ is selective then the following are equivalent.
\begin{enumerate}
\item\label{C.S.2:1}
There exists a nontrivial central sequence in $M$.
\item\label{C.S.2:2}
There exists a nontrivial $\cV$-central sequence in $M$.
\end{enumerate}
\end{coro}

\begin{proof}
Assume $(a_n)_{n \in \bbN}$ is a nontrivial central sequence.
By passing to a subsequence, we
may assume there is  $\e>0$
such that $\inf_{\ld \in \bbC} \|a_n - \ld\| \geq \e$ for all $n$.
This sequence is clearly a nontrivial $\cV$-central sequence
for any
nonprincipal ultrafilter $\cV$.
The converse implication is an immediate
consequence of Theorem~\ref{T.S.2.0}.
\end{proof}

Let $M$ be a type II$_1$ factor with (unique) tracial state~$\ta$.
Let $\|\cdot\|_2$ be the standard
$L^2$-norm on $M$, defined by $\| a \|_2 = \sqrt{\tau (a^*a)}$.
A bounded sequence $(a_n)$ in $M$ is {\emph{tracially central}}
if $\lim_n \|[b,a_n]\|_2=0$ for every $b\in M$.
(This, rather than that of Definition~\ref{D:Central},
is the usual definition of a central sequence in this context.)
The {\emph{tracial ultrapower}} of a II$_1$
factor is the ultrapower of the metric structure $(M, \, \|\cdot\|_2)$
in the sense of~\cite{BYBHU}.
Analogues of Theorem~\ref{T.S.2.0} and
Corollary~\ref{C.S.2}
for tracial ultrapowers of II$_1$ factors
could be proved by mimicking the
above proofs.
However, in this context the assumption that $\cV$ is a
selective
ultrafilter can be dropped.
In fact, for every nonprincipal ultrafilter $\cV$ on $\bbN$,
the commutant of $M$ in $M^{\cV}$ is trivial
if and only if $M$ has no
nontrivial central sequences.
By an observation of McDuff
(\cite[remark after Lemma~5]{McDuff:Central})
this follows by a diagonalization argument from the fact that the metric
$\|\cdot\|_2$ on a II$_1$ factor is separable.
Similarly, the analogues of Theorem~\ref{T.S.2.0}
and Corollary~\ref{C.S.2}
hold when  $M$ is a separable C*-algebra
and $\cV$ is any nonprincipal
ultrafilter on~$\bbN$.

On the other hand,
Theorem~\ref{T1} and Theorem~\ref{T.S.3} below
imply that some assumption on $\cV$
in Theorem~\ref{T.S.2.0} and Corollary~\ref{C.S.2}
is necessary in the case
when $M=\BH$.

\section{No nontrivial central sequences}

Let $M$ be a von Neumann algebra with center ${\mathcal{Z}} (M)$.
In \cite[Proposition~3.1]{She:Divisible}, David Sherman proved that
a sequence $(a_n)_{n\in \bbN}$ in $M$ is norm central
if and only if
$\lim_{n\to \infty} \inf_{z\in {\mathcal{Z}} (M)} \| a_n - z \| = 0$.
In particular,
if $M$ is a factor then it has no nontrivial central sequences.
Note that we are referring to {\textbf{norm}} central sequences,
even if $M$ is a II$_1$-factor.
For the reader's convenience we provide a self-contained proof of the
special case of Sherman's result
needed in the proof of Theorem~\ref{T2}.

\begin{thm}\label{T.S.3}
There are no nontrivial central sequences in $\BH$.
\end{thm}

The proof of Theorem~\ref{T.S.3} will be given after two lemmas.

\begin{lem}\label{L:6}
Let  $c \in \BH$ be a selfadjoint operator that is not a scalar.
Set $\dt = \inf_{\ld \in \C} \| c - \ld \|$.
Then $\dt > 0$,
and for every $\ep > 0$ there is a rank one partial
isometry $s \in \BH$ such that $\| [ s, c ] \| > 2 \dt - \ep$.
\end{lem}

\begin{proof}
We denote the spectrum of $c$ by $\spec (c)$.
Set $\ld_1 = \inf (\spec (c))$ and $\ld_2 = \sup (\spec (c))$.
By considering the isomorphism between $C (\spec (c))$
and the C*-subalgebra of $\BH$ generated by $c$ and $1$, we
see that $\dt = \frac{1}{2} (\ld_2 - \ld_1) > 0$.
Set
\[
\ep_0 = \frac{\ep}{4 \| c \| + 2}.
\]
Choose $\ep_1 > 0$ so small that
if $\af \in \R$ satisfies $| \af - 1 | < \ep_1 \dt^{-1}$,
then $| \af^{-1} - 1 | < \ep_0$.
We also require $\ep_1 \leq \min (\ep_0, \dt \ep_0)$.

Choose $\xi_1, \et \in H$ with
\[
\| \xi_1 \| = \| \et \| = 1,
\,\,\,\,\,\,
\| c \xi_1 - \ld_1 \xi_1 \| < \ep_1,
\andeqn \| c \et - \ld_2 \et \| < \ep_1.
\]
Set
\[
\mu = \et - \langle \et, \xi_1 \rangle \xi_1
\andeqn
\xi_2 = \| \mu \|^{-1} \mu.
\]
Then
$\| \xi_2 \| = 1$ and $\langle \xi_1, \xi_2 \rangle = 0$.

We need to
estimate $\|c\xi_2-\lambda_2\xi_2\|$.
We begin as follows, using
$\langle \xi_1, c \et \rangle = \langle c \xi_1, \et \rangle$
at the second step:
\begin{align*}
| (\ld_2 - \ld_1) \langle \xi_1, \et \rangle |
& = | \langle \xi_1, \ld_2 \et \rangle
          - \langle \ld_1 \xi_1, \et \rangle |
      \\
& = | \langle \xi_1, \, \lambda_2\et - c\et \rangle
          + \langle c\xi_1-\lambda_1\xi_1, \, \et\rangle|
      \\
& \leq \| \xi_1 \| \cdot \| \ld_2 \et - c \et \|
          + \| \ld_1 \xi_1 - c \xi_1 \| \cdot \| \et \|
   < 2 \ep_1.
\end{align*}
It follows that
\[
\| \mu - \et \|=|\langle \xi_1, \et \rangle|
   < \frac{2 \ep_1}{\ld_2 - \ld_1} = \ep_1 \dt^{-1},
\]
so $\big| 1 - \| \mu \| \big| < \ep_1 \dt^{-1}$, and
\[
\| \xi_2 - \et \|
   \leq \big\| \| \mu \|^{-1} \mu - \mu \big\| + \| \mu - \et \|
   < \ep_0 + \ep_1 \dt^{-1}
   \leq 2 \ep_0.
\]
Therefore
\[
\| c \xi_2 - \ld_2 \xi_2 \|
   < 2 \ep_0 \| c \| + \| c \et - \ld_2 \et \| + 2 \ep_0 | \ld_2 |
   \leq (4 \| c \| + 1 ) \ep_0.
\]

Now let $s \in \BH$ be the partial isometry
such that $s \xi_1 = \xi_2$
and $s \xi = 0$ whenever $\langle \xi, \xi_1 \rangle = 0$.
Then
\[
\| s c \xi_1 - \ld_1 \xi_2 \|
   \leq \| s \| \cdot \| c \xi_1 - \ld_1 \xi_1 \|
   < \ep_1
   \leq \ep_0.
\]
So
\begin{align*}
\| [s, c] \| & \geq \| s c \xi_1 - c s \xi_1 \|
   = \| s c \xi_1 - c \xi_2 \|
   \\
 & \geq (\ld_2 - \ld_1)
         - \| \ld_1 \xi_2 - s c \xi_1 \| - \| \ld_2 \xi_2 - c \xi_2 \|
   \\
 & > 2 \dt - \ep_0 - (4 \| c \| + 1 ) \ep_0
   \geq 2 \dt - \ep.
\end{align*}
This completes the proof.
\end{proof}

\begin{lem}\label{L:DistC}
Let $A$ be a unital \ca\  and let $a \in A$.
Then
\[
\inf_{\ld \in \C} \| a - \ld \|
   \geq \tfrac{1}{2} ( \| a \| - \| p a p \| )
\]
for every nonzero projection $p$ in $A$.
\end{lem}

\begin{proof}
Let $\ld \in \C$.
We have
\[
\| a - \ld \| \geq \| a \| - | \ld |.
\]
Also, using $p \neq 0$ at the third step,
\[
\| a - \ld \|
  \geq \| p (a - \ld) p \|
  = \| p a p - \ld p \|
  \geq | \ld | - \| p a p \|.
\]
Therefore
\[
\| a - \ld \|
  \geq \tfrac{1}{2} ( \| a \| - | \ld | )
         + \tfrac{1}{2} ( | \ld | - \| p a p \| )
  = \tfrac{1}{2} ( \| a \| - \| p a p \| ),
\]
as desired.
\end{proof}

\begin{proof}[Proof of Theorem~\ref{T.S.3}]
Using the decomposition of an operator~$a$
as $a = \frac{1}{2} ( a + a^* ) + \frac{1}{2} ( a - a^* ),$
with $\frac{1}{2} ( a + a^* )$ and $- \frac{i}{2} ( a - a^* )$
selfadjoint,
we need only prove that there are no nontrivial selfadjoint central
sequences.
Let $(a_j)_{j \in \N}$ be a norm bounded sequence of
selfadjoint elements of $\BH$ (not necessarily central) which is not
trivial in the sense of Definition~\ref{D:Central}. We prove that
$(a_j)_{j \in \N}$ is not central. It suffices to find a subsequence
which is not central. By passing to a subsequence, we may assume
\[
\inf_{j \in \N} \inf_{\ld \in \bbC} \| a_j - \ld \| > 0.
\]
In particular, our sequence $(a_j)_{j \in \N}$
has no subsequence
which converges in norm to any element of $\C \cdot 1$.

We will find  $\ep > 0$, an element $b \in \BH$,
and a subsequence $(a_{m (j)})_{j \in \N}$ of $(a_j)_{j \in \N}$,
such that $\| [ a_{m (j)}, \, b ] \| > \ep$ for all $j \in \N$.
This will show that $(a_j)_{j \in \N}$ is not central,
and prove the theorem.

Let $(r_n)_{n \in \N}$ be a strictly increasing
sequence of finite rank \pj s such that $\limi{n} r_n = 1$
in the strong
operator topology, and with $r_0 = 0$.

First suppose that there are $n \in \N$, a number $\dt > 0$,
and a subsequence
$(a_{l (j)})_{j \in \N}$ of $(a_j)_{j \in \N}$, such that
\[
\inf_{j \in \N} \inf_{\ld \in \C} \| r_n a_{l (j)} r_n - \ld r_n \|
   \geq \dt.
\]
Since $r_n$ has finite rank,
there is a further subsequence
$(a_{m (j)})_{j \in \N}$ of $(a_{l (j)})_{j \in \N}$
such that $c = \limi{j} r_n a_{m (j)} r_n$
exists.
Then also $\inf_{\ld \in \C} \| c - \ld r_n \| \geq \dt$.
Lemma~\ref{L:6} provides $s \in r_n \BH r_n$
such that $\| [s, c] \| > \frac{3}{2} \dt$.
Then
\[
\| [s, \, a_{m (j)}] \| = \| [s, \, r_n a_{m (j)} r_n ] \| > \dt
\]
for all sufficiently large~$j$.
Dropping initial terms of the
subsequence $(a_{m (j)})_{j \in \N}$,
we obtain the required subsequence with $b = s$ and $\ep = \dt$.

Accordingly, we may now assume that
\[
\lim_{j \to \I} \inf_{\ld \in \C} \| r_n a_j r_n - \ld r_n \| = 0
\]
for all $n \in \N$.

Set $M = \sup_{j \in \N} \| a_j \|$.
Then
\[
\lim_{j \to \I} \inf_{\ld \in [-M, \, M]} \| r_n a_j r_n - \ld r_n \|
    = 0
\]
for all $n \in \N$.
In particular, there are numbers $\ld_{0, j} \in [-M, \, M]$ such that
\[
\lim_{j \to \I} \| r_0 a_j r_0 - \ld_{0, j} r_0 \|
    = 0.
\]
By compactness, there are $\ld_0 \in [-M, \, M]$
and a strictly increasing
sequence $(l_0 (j))_{j \in \N}$
such that $\lim_{j \to \I} \ld_{0, l_0  (j)} = \ld_0$.
Then
\[
\lim_{j \to \I} \| r_0 a_j r_0 - \ld_0 r_0 \|
    = 0.
\]
Similarly, there are $\ld_{1, j} \in [-M, \, M]$ such that
\[
\lim_{j \to \I} \| r_1 a_{l_0 (j)} r_1 - \ld_{1, j} r_1 \|
    = 0,
\]
and then there are $\ld_1 \in [-M, \, M]$
and a subsequence $(l_1 (j))_{j \in \N}$ of $(l_0 (j))_{j \in \N}$
such that
\[
\lim_{j \to \I} \| r_1 a_{l_1 (j)} r_1 - \ld_1 r_1 \|
    = 0.
\]
We proceed inductively, obtaining numbers $\ld_n \in [-M, \, M]$ and
subsequences $(l_n (j))_{j \in \N}$ of $(l_{n - 1} (j))_{j \in \N}$
such that
\[
\lim_{j \to \I} \| r_n a_{l_n (j)} r_n - \ld_n r_n \|
    = 0.
\]
Setting $l (j) = l_j (j)$, we then get
\[
\lim_{j \to \I} \| r_n a_{l (j)} r_n - \ld_n r_n \|
    = 0
\]
for all $n \in \N$.

Clearly $\ld_0 = \ld_1 = \cdots$.
Subtracting this common value from each $a_j$
does not change the conditions on $(a_j)_{j \in \N}$
or the existence of the required subsequence,
so \wolog\  $\limi{j} \| r_n a_{l (j)} r_n \| = 0$ for all~$n$.

Suppose now that there is $n \in \N$
such that $\| r_n a_{l (j)} (1 - r_n) \|$ does
not converge to~$0$ as $j \to \I$.
Then there are $\rh > 0$ and a subsequence
$(a_{m (j)})_{j \in \N}$ of $(a_{l (j)})_{j \in \N}$
such that $\| r_n a_{m (j)} (1 - r_n) \| > \rh$ for all $j \in \N$.
With $b = r_n$, and using $r_n (1 -  r_n) = 0$ at the second step,
we have
\[
\| [ b, a_{m (j)} ] \|
   \geq \| r_n [ b, a_{m (j)} ] (1 - r_n) \|
   = \| r_n a_{m (j)} (1 - r_n) \|
   > \rh
\]
for all $j \in \N$.
Thus,
we have a subsequence $(a_{m (j)})_{j \in \N}$ of the required type
with $b = r_n$ and $\ep = \rh$.

Since $\| r_n a_j (1-r_n)\|=\|(1-r_n) a_j r_n \|$, we may assume
\wolog\ that
\[
\limi{j} \| r_n a_j (1 - r_n) \|
    = \limi{j} \| (1 - r_n) a_j r_n \| = 0
\]
for all $n \in \N$.
Combining these with
$\limi{j} \| r_n a_j r_n \| = 0$ for
all~$n$, we have reduced to consideration of the case
\[
\limi{j} \| a_j - (1 - r_n) a_j (1 - r_n) \| = 0
\]
for all $n \in \N$.

Since by assumption $(a_j)_{j \in \N}$
does not converge to $0$ in norm,
there are $\dt > 0$
and a subsequence $(a_{l (j)})_{j \in \N}$ of $(a_j)_{j \in \N}$
such that $\inf_{j \in \N} \| a_{l (j)} \| \geq \dt$.
Passing to this subsequence,
without loss of generality $\inf_{j \in \N} \| a_j \| \geq \dt$.
We now construct recursively a
subsequence $(a_{m (j)})_{j \in \N}$ of $(a_j)_{j \in \N}$,
and an increasing
sequence $(n (j))_{j \in \N}$ with $n (0) = 0$, such that the elements
\begin{equation}\label{Eq:yj}
y_j = \big( r_{n (j + 1)} - r_{n (j)} \big) a_{m (j)}
    \big( r_{n (j + 1)} - r_{n (j)} \big)
\end{equation}
satisfy
\[
\inf_{\ld \in \C}  \big\| y_j - \ld (r_{n (j + 1)} - r_{n (j)})  \big\|
   > \tfrac{1}{4} \dt
\]
for all $j \in \N$.
We repeatedly use the observation that
$\limi{n} \| r_n x r_n \| = \| x \|$ for all $x \in \BH$.

Begin by choosing $n (0) = 0$, so that $r_{n (0)} = 0$.
Choose $m (0) \in \N$
such that $\| r_1 a_{m (0)} r_1 \| < \tfrac{1}{4} \dt$.
Now use $\| a_{m (0)} \| \geq \dt$
to choose $n (1)$ so large that the element
$y_0 = r_{n (1)} a_{m (0)} r_{n (1)}$
satisfies $\| y_0 \| > \tfrac{3}{4} \dt$.
Apply
Lemma~\ref{L:DistC} with $A = r_{n (1)} \BH r_{n (1)}$,
with $a = y_0$, and
with $p = r_1$, to get
\[
\inf_{\ld \in \C} \| y_0 - \ld r_{n (1)} \|
   > \tfrac{1}{2} \left( \tfrac{3}{4} \dt - \tfrac{1}{4} \dt \right)
   = \tfrac{1}{4} \dt.
\]

Given $m (j)$ and $n (j + 1)$,
choose $m (j + 1) > m (j)$ so large that,
with
\[
x = \big( 1 - r_{n (j + 1) + 1} \big)
    a_{m (j + 1)} \big( 1 - r_{n (j + 1) + 1} \big),
\]
we have $\| a_{m (j + 1)} - x \| < \tfrac{1}{6} \dt$.
Then choose $n (j + 2)$
so large that
\[
\| r_{n (j + 2)} x r_{n (j + 2)} \|
  > \| x \| - \tfrac{1}{6} \dt.
\]
Set
\[
y_{j + 1}
  = \big( r_{n (j + 2)} - r_{n (j + 1)} \big) a_{m (j + 1)}
          \big( r_{n (j + 2)} - r_{n (j + 1)} \big).
\]
Then
\begin{align*}
\| y_{j + 1} \|
 & \geq \big\| \big( r_{n (j + 2)} - r_{n (j + 1) + 1} \big)
    a_{m (j + 1)} \big( r_{n (j + 2)} - r_{n (j + 1) + 1} \big) \big\|
      \\
 & = \| r_{n (j + 2)} x r_{n (j + 2)} \|
   > \| x \| - \tfrac{1}{6} \dt
   > \| a_{m (j + 1)} \| - \tfrac{1}{6} \dt - \tfrac{1}{6} \dt
   \geq \tfrac{2}{3} \dt.
\end{align*}
Also,
\[
\big\| r_{n (j + 1) + 1} y_{j + 1} r_{n (j + 1) + 1} \big\|
   \leq \big\| r_{n (j + 1) + 1} (a_{m (j + 1)} - x)
             r_{n (j + 1) + 1} \big\|
   < \tfrac{1}{6} \dt.
\]
Apply Lemma~\ref{L:DistC} with
\[
A = \big( r_{n (j + 2)} - r_{n (j + 1)} \big)
          \BH \big( r_{n (j + 2)} - r_{n (j + 1)} \big),
\]
with $a = y_{j + 1}$, and with $p = r_{n (j + 1) + 1} - r_{n (j + 1)}$,
to get
\[
\inf_{\ld \in \C}
        \big\| y_{j + 1} - \ld (r_{n (j + 2)} - r_{n (j + 1)} ) \big\|
   > \tfrac{1}{2} \left( \tfrac{2}{3} \dt - \tfrac{1}{6} \dt \right)
   = \tfrac{1}{4} \dt,
\]
as desired.
This completes the construction of $(a_{m (j)})_{j \in \N}$
and $(n (j))_{j \in \N}$.

Now let $y_j$ be as in~(\ref{Eq:yj}) for $j \in \N$.
Lemma~\ref{L:6} provides
\[
s_j
  \in \big( r_{n (j + 1)} - r_{n (j)} \big)
         \BH \big( r_{n (j + 1)} - r_{n (j)} \big)
\]
such that $\| [ s_j, y_j ] \| > \tfrac{1}{4} \dt$ and $\| s_j \| = 1$.
The
series $s = \sum_{j = 0}^{\I} s_j$
converges in the strong operator topology,
and for $j \in \N$ we have
\[
\| [s, a_{m (j)}] \|
   \geq \big\| \big( r_{n (j + 1)} - r_{n (j)} \big)
          [s, a_{m (j)}] \big( r_{n (j + 1)} - r_{n (j)} \big) \big\|
   = \| [ s_j, y_j ] \|
   > \tfrac{1}{4} \dt.
\]
Thus, the subsequence $(a_{m (j)})_{j \in \N}$
satisfies the required condition
with $b = s$ and $\ep = \tfrac{1}{4} \dt$.
\end{proof}

The following is an immediate consequence of Corollary~\ref{C.S.2} and
Theorem~\ref{T.S.3}.

\begin{coro} \label{C.S.1}
If $\cV$ is a selective ultrafilter then $F_{\cV} (\BH) = \bbC$. \qed
\end{coro}

\section{Flat ultrafilters}\label{Sec:2}

\begin{notation}\label{N:Up}
By $f\colon \bbN\nearrow\bbN$ we mean that $f$ is a strictly increasing
function from $\bbN$ to $\bbN$ such that $f(0)>0$.
\end{notation}

For such $f$ and
nonincreasing $h\colon \bbN\to [0,1]$
the assertion $\|h - h \circ f\|_{\infty}  \leq \e$
is equivalent to stating that the variation of $h$ on any
interval of the form $\N \cap [j, f (j)]$ is at most~$\e$.

\begin{definition}\label{D:Flat}
An ultrafilter $\cV$ on $\bbN$ is {\emph{flat}} if
there are nonincreasing functions
$h_n \colon \bbN \to [0, 1]$, for $n \in \N$, such that:
\begin{enumerate}
\item\label{D:Flat:1}
$h_n (0) = 1$ for all $n \in \N$.
\item\label{D:Flat:2}
$\lim_{j \to \infty} h_n (j) = 0$ for all $n \in \N$.
\item\label{D:Flat:3}
For every $f \colon \bbN \nearrow \bbN$,
we have
$\lim_{n \to \cV} \| h_n - h_n \circ f\|_{\infty} = 0$.
\end{enumerate}
\end{definition}

\begin{thm}\label{T.Flat.1}
Flat ultrafilters exist.
\end{thm}

We need a lemma.

\begin{lemma}\label{L.Flat.1}
Let $f\colon \bbN \nearrow \bbN$.
Let $n \in \N$ with $n > 0$,
let $m_0 = 0$,
and suppose $m_{l + 1}\geq f (m_l)$ for $0 < l \leq n$.
Set
\[
h = \sum_{l = 0}^{n} \frac{n - l}{n} \chi_{\N \cap [m_l, \, m_{l + 1})}.
\]
Then $\| h - h \circ f \|_{\infty} \leq 1/n$.
\end{lemma}

\begin{proof}
Fix $j\in \bbN$.
If $j \geq m_{n + 1}$ then $h (j) = 0 = h \circ f (j)$.
Otherwise there is $l$ such that $m_l\leq j<m_{l+1}$.
Then
$f (j) < f (m_{l + 1}) \leq m_{l + 2}$ (writing $m_{n + 2} = \infty$).
Since $h$ is nonincreasing,
\[
\frac{n - l}{n}
 = h (j)
 \geq h \circ f (j)
 \geq \frac{n - l - 1}{n}.
\]
The required estimate is now clear.
\end{proof}

\begin{proof}[Proof of Theorem~\ref{T.Flat.1}]
Let $\bbF$ be the countable set of all nonincreasing functions
$h \colon \bbN \to \Q \cap [0, 1]$ that are eventually zero
and such that $h (0) = 1$.
We start by constructing an ultrafilter $\cV$ on $\bbF$.
For $f\colon \bbN\nearrow\bbN$ and
$\e>0$ let
\[
X_{f,\e}=\{h\in \bbF\colon \|h-h\circ f\|_{\infty} \leq \e\}.
\]
By Lemma~\ref{L.Flat.1} this set is infinite.
On the other hand,
\[
X_{f, \e} \cap X_{g, \delta}
 \supseteq X_{\max(f, g), \, \min(\e, \delta)}.
\]
Therefore
the collection of all $X_{f, \e}$,
for $f \colon \bbN \nearrow\bbN$ and $\e > 0$,
has the  finite intersection property.
Let $\cW$ be any ultrafilter which extends this collection.

Let $k \colon \bbN \to \bbF$ be a bijection,
and set $\cV = \{ A \subseteq \bbN \colon k (A) \in \cW \}$,
which is an ultrafilter on $\bbN$.
We claim that $\cV$ is flat.
The functions $h_n$ required in the definition are given by
$h_n = k (n)$ for $n \in \N$.
Conditions (\ref{D:Flat:1}) and~(\ref{D:Flat:2})
in Definition~\ref{D:Flat} are immediate.
For Condition~(\ref{D:Flat:3}),
let $f \colon \bbN \nearrow \bbN$
and let $\e > 0$.
Then $Y = k^{-1} (X_{f, \e}) \in \cV$,
and for $n \in Y$ we have $h_n \in X_{f, \e}$
by construction, so that $\| h_n - h_n \circ f\|_{\infty} \leq \e$.
This proves~(\ref{D:Flat:3}) in Definition~\ref{D:Flat}.
\end{proof}

\section{Nontrivial relative commutants}

The present section is devoted to the proof of the following result.

\begin{thm}\label{T.Comm.1}
If $\cV$ is a flat ultrafilter then
$F_{\cV} (\BH)\neq \bbC$.
\end{thm}

\begin{notation}\label{N:ah}
Fix an orthonormal basis $(\xi_n)_{n \in \N}$ for our separable
infinite-dimensional complex Hilbert space $H$, and let $e_n$ be the
orthogonal projection onto $\C \xi_n$. Let $\bbD$ be the set of all
nonincreasing functions $h \colon \N \to [0, 1]$
such that $h (0) = 1$ and $\lim_{n \to \infty} h (n) = 0$.
For $h\in \bbD$ define a
compact operator $a_h$ (with $\| a_h \| = 1$ since $h (0) = 1$) by
\[
a_h = \sum_{j = 0}^{\infty} h (j) e_j.
\]
\end{notation}

\begin{notation}\label{N:DE}
Let $\vec{E} = (E_n)_{n \in \N}$ be a family of
closed orthogonal subspaces of $H$
such that $H = \bigoplus_{n = 0}^{\infty} E_n$.
Let $\calD \big( \vec{E} \big)$ be the von Neumann algebra
\[
\{a\in \BH \colon {\mbox{$a E_n \subseteq E_n$ for all $n \in \N$}} \}.
\]
For $f \colon \bbN \nearrow \bbN$
(as in Notation~\ref{N:Up}),
and with $(\xi_n)_{n \in \N}$ as in Notation~\ref{N:ah},
let $f^n$ be the composite
$f \circ f \circ \cdots \circ f$ (with $n$ terms),
and take $f^0$ to be the constant function with value~$0$.
Define $\vec{E}^f$ by
\[
\vec{E}^f_n = \SPAN \{ \xi_j \colon f^n (0) \leq j < f^{n + 1} (0) \},
\]
and set $\calD (f) = \calD \big( \vec{E}^f \big)$.
\end{notation}

\begin{lemma}\label{L.Comm.1}
Adopt Notation \ref{N:ah} and Notation~\ref{N:DE}.
If $h \in \bbD$, $f \colon \bbN \nearrow \bbN$,
and $\| h - h \circ f \|_{\infty}\leq \e$, then for every
$b \in \calD (f)$ we have $\| [a_h, b] \| \leq 2 \e \|b\|$.
\end{lemma}

\begin{proof}
Let $q_n$ be the orthogonal projection onto $E_n^f$.
We can write
\[
a_h = \sum_{n = 0}^{\infty} q_n a_h q_n.
\]
Define
\[
y = \sum_{n = 0}^{\infty} h (f^n (0)) q_n.
\]
(Both series converge in norm because $\lim_{n \to \infty} h (n) = 0$.)
For any $n \in \N$
and for $f^n (0) \leq k < f^{n + 1} (0)$, we have
\[
h (f^n (0)) \geq h (k) \geq h (f^{n + 1} (0)) \geq h (f^n (0)) - \e,
\]
so $\big\| q_n a_h q_n - h (f^n (0)) q_n \big\| \leq \ep$.
Therefore $\| a_h - y \| \leq \ep$.
Since $y$ is a central element of $\calD (f)$,
the conclusion follows.
\end{proof}

\begin{lem}\label{L:TNorm}
Let $A$ be a unital \ca, let $e, f \in A$ be orthogonal \pj s,
and let $a \in A$.
Then
\[
\| e a e + e a f + f a e \| \leq 2 \| a \|.
\]
\end{lem}

\begin{proof}
We have
\[
e a e + e a f + f a e
  = (e + f) a (e + f) - f a f,
\]
and $\| (e + f) a (e + f) \|, \, \| f a f \| \leq \| a \|$.
\end{proof}

Examples using $2 \times 2$ matrices show that
it is not possible to replace the constant~$2$ in Lemma~\ref{L:TNorm}
by~$1$, even if $a$ is selfadjoint.

The use of `stratification' of $\BH$ into von Neumann algebras
$\calD (g_l)$ as given in Lemma~\ref{L.Comm.2} below
resembles the use in
\cite[Lemma~3.1]{Fa:Calkin}, and the following lemma is a minor
improvement to \cite[Lemma~1.3]{Fa:Calkin}.

\begin{lem}\label{L:4}
Let $F = \{ a_1, a_2, \ldots, a_k \} \subset \BH$
be finite, and let $\dt > 0$.
Then there exist $g_0, g_1 \colon \N \nearrow \N$ and
decompositions $a_j = a_j^{(0)} + a_j^{(1)} + c_j$
for $j = 1, 2, \ldots, k$,
such that for $j = 1, 2, \ldots, k$ we have:
\begin{enumerate}
\item\label{L:4:a0Comm}
$a_j^{(0)} \in \calD (g_0)$.
\item\label{L:4:a1Comm}
$a_j^{(1)} \in \calD (g_1)$.
\item\label{L:4:Bound}
$\big\| a_j^{(0)} \big\|, \big\| a_j^{(1)} \big\| \leq 2 \| a_j \|$.
\item\label{L:4:Cpt}
$c_j$ is compact.
\item\label{L:4:Small}
$\| c_j \| < \dt$.
\end{enumerate}
\end{lem}

\begin{proof}
Let $p_n$ be the
orthogonal \pj\  onto
$\spn (\{ \xi_0, \xi_1, \ldots, \xi_{n - 1} \})$.
Thus $p_0 = 0$.
Also choose $\rh_0, \rh_1, \ldots > 0$
such that $2 \sum_{n = 0}^{\I} \rh_{n + 1} \leq \dt$.

We claim that there is
a strictly increasing function $f \colon \N \to \N$ such
that $f (0) = 0$
and such that for every $n \in \N$ and every $a \in F$, we
have
\begin{equation}\label{L:4:Est}
\big\| (1 - p_{f (n + 1)} ) a p_{f (n)} \big\|
 < \rh_n \andeqn \big\| p_{f (n)} a (1 - p_{f (n + 1)} ) \big\|
 < \rh_n.
\end{equation}
(For $n = 0$ the condition is vacuous because $p_0 = 0$.)
We construct $f$ recursively.
Start by taking $f (0) = 0$.
Given $f (n)$, use
compactness of $p_{f (n)} a$ and $a p_{f (n)}$,
finiteness of $F$, and the fact
that $(p_m)_{m \in \N}$ is an approximate identity for $\cK(H)$,
to choose $m > f (n)$ such that
\[
\big\| (1 - p_m ) a p_{f (n)} \big\|
 < \rh_n \andeqn \big\| p_{f (n)} a (1 - p_m ) \big\|
 < \rh_n
\]
for all $a \in F$.
Then set $f (n + 1) = m$.
This proves the claim.

For $n \in \N$, we now set $q_n = p_{f (n + 1)} - p_{f (n)}$.
Since $p_{f (0)} = 0$,
the series $\sum_{n = 0}^{\I} q_n$ converges to~$1$ in the strong
operator topology.

For $j = 1, 2, \ldots, k$, define,
with convergence in the strong operator topology,
\[
a_j^{(0)}
  = \sum_{n = 0}^{\I} \big( q_{2 n} a_j q_{2 n}
                             + q_{2 n} a_j q_{2 n + 1}
                             + q_{2 n + 1} a_j q_{2 n} \big)
\]
and
\[
a_j^{(1)}
  = \sum_{n = 0}^{\I} \big( q_{2 n + 1} a_j q_{2 n + 1}
                             + q_{2 n + 1} a_j q_{2 n + 2}
                             + q_{2 n + 2} a_j q_{2 n + 1} \big).
\]
The $n$th term in the series for
$a_j^{(0)}$ is in $(q_{2 n} + q_{2 n + 1}) \BH (q_{2 n} + q_{2 n + 1})$,
and the $n$th term in the series for
$a_j^{(1)}$ is
in $(q_{2 n + 1} + q_{2 n + 2}) \BH (q_{2 n + 1} + q_{2 n + 2})$.
Accordingly,
if for $j \in \N$
we set $g_0 (j) = f (2 j + 2)$ and $g_1 (j) = f (2 j + 1)$,
then $g_0, g_1 \colon \N \nearrow \N$
and parts (\ref{L:4:a0Comm}) and~(\ref{L:4:a1Comm}) are satisfied.
Part~(\ref{L:4:Bound}) follows from Lemma~\ref{L:TNorm}.

The estimates~(\ref{L:4:Est}) give, for every $n \in \N$,
\[
\big\| q_n a_j (1 - p_{f (n + 2)}) \big\|
   = \big\| q_n p_{f (n + 1)} a_j (1 - p_{f (n + 2)}) \big\|
   < \rh_{n + 1},
\]
and similarly
\[
\big\| (1 - p_{f (n + 2)}) a_j q_n \big\|
   < \rh_{n + 1}.
\]
Therefore the series
\[
\sum_{n = 0}^{\I} \big[ q_n a_j (1 - p_{f (n + 2)})
                          + (1 - p_{f (n + 2)}) a_j q_n \big]
\]
converges in norm to a compact operator $c_j$ with
\[
\| c_j \| < 2 \sum_{n = 0}^{\I} \rh_{n + 1} \leq \dt.
\]
This is parts (\ref{L:4:Cpt}) and~(\ref{L:4:Small}).
Also, $a_j^{(0)} + a_j^{(1)} + c_j = a_j$ is clear.
\end{proof}

\begin{lemma} \label{L.Comm.2}
Let $\cV$ be an arbitrary ultrafilter on $\bbN$.
For $\bfa \in \BH^{\cV}$ the following are equivalent:
\begin{enumerate}
\item \label{L.Comm.2.1}
$\bfa \in \BH' \cap \BH^{\cV}$.
\item \label{L.Comm.2.2}
$\bfa \in \bigcap_{f \colon \bbN \nearrow \bbN}
                 \big[ \calD (f)' \cap \BH^{\cV} \big]$.
\end{enumerate}
\end{lemma}

\begin{proof}
The implication
from \eqref{L.Comm.2.1} to~\eqref{L.Comm.2.2} is trivial.

Assume~\eqref{L.Comm.2.2} and fix $\bfb\in \BH$.
Fix $\delta>0$.
By
Lemma~\ref{L:4} we can find $g_0, g_1 \colon \bbN \nearrow \bbN$
and a decomposition $\bfb=\bfb_0+\bfb_1+\bfc$ such
that $\bfb_j \in \calD (g_j)$ for $j = 0, 1$
and $\|\bfc\| \leq \tfrac{1}{2} \delta$.
Thus
$[\bfa, \bfb] = [\bfa, \, \bfb_0 + \bfb_1 + \bfc] = [\bfa, \bfc]$
and therefore $\|[\bfa, \bfb]\| \leq \delta \|\bfa\|$.
Since
$\bfb\in \BH$ and $\delta>0$ were arbitrary,
$\bfa \in \BH' \cap \BH^{\cV}$.
\end{proof}

\begin{proof}[Proof of Theorem~\ref{T.Comm.1}]
Fix a sequence $(h_n)_{n \in \N}$ of functions
witnessing the flatness of $\cV$.
Let $a_n = a_{h_n} = \sum_{j = 0}^{\infty} h_n (j) e_n$,
as in Notation~\ref{N:ah}.
Fix $f\colon \bbN\nearrow\bbN$.
Since $\lim_{n \to \cV} \| h_n - h_n \circ f \|_{\infty} = 0$,
by Lemma~\ref{L.Comm.1} the sequence $(a_n)_{n \in \N}$
is a representing sequence
of an element $\bfa$ of $\calD(f)' \cap \BH^{\cV}$.
Since $f\colon \bbN\nearrow\bbN$ was arbitrary,
by Lemma~\ref{L.Comm.2} we have $\bfa \in \BH' \cap \BH^{\cV}$.

Therefore $(a_n)_{n \in \N}$ is a $\cV$-central sequence.
Since each $a_n$ is compact and
has norm one, this sequence is nontrivial.
\end{proof}

\section{Concluding remarks}

The following is what remains of Kirchberg's question.

\begin{question} \label{Q1}
Does there exist a  nonprincipal ultrafilter $\cV$ on $\bbN$
such that $F_{\cV}(\BH)=\bbC$?
\end{question}

By our Theorem~\ref{T2}, the Continuum Hypothesis implies a positive
answer, but the question is whether such an ultrafilter can be
constructed in ZFC.
A `typical' statement independent from ZFC is
decided by the Continuum Hypothesis or a strengthening such as
Jensen's diamond principle in one way and by Martin's Axiom or a
strengthening such as the Proper Forcing Axiom in another way.
(See
\cite[Chapter II]{Ku:Book} for an introduction to Martin's Axiom.)
An
example in theory of operator algebras is the statement `the Calkin
algebra has an outer automorphism,' which follows from the Continuum
Hypothesis (\cite{PhWe:Calkin}) and is incompatible with a
consequence of the Proper Forcing Axiom (\cite{Fa:Calkin}). This,
however, is not the case with Question~\ref{Q1}.
It is well-known that (a rather weak form of) Martin's Axiom implies
the existence of selective ultrafilters, and therefore the existence
of $\cU$ such that $F_{\cU}(\BH)=\bbC$. A closer look at the proof of
Proposition~\ref{P.selective} reveals that it goes through when the
Continuum Hypothesis is weakened to the assertion that for every
family $\cF\subseteq [\bbN]^{\infty}$ such that the intersection
of any finitely many sets in $\cF$ is infinite,
and such that $|\cF|<2^{\aleph_0}$,
there is $B\in [\bbN]^{\infty}$ such that
$B\setminus A$ is finite for all $A\in \cF$.
This assertion (known as
$\mathfrak{p}=2^{\aleph_0}$) is an easy consequence of Martin's Axiom.
(See \cite[Section~7]{Bla:Cardinal}.)

By a result of Kunen (\cite{Ku:Some}),
if ZFC is consistent then so is the theory `ZFC + there are no
selective ultrafilters'. However, in Kunen's model there exists an
ultrafilter $\cV$ such that $F_{\cV} (\BH) = \bbC$. An ultrafilter
$\cV$ is a {\emph{P-point}} if for every $g\colon \bbN\to \bbN$ there
is $A\in \cV$ such that $g$ is either constant or finite-to-one on
$A$. In~\cite{FaSt:Flat} it is proved that if $\cV$ is a P-point then
$F_{\cV}(\BH)=\bbC$. While P-points exist in Kunen's model, Shelah
has proved that if ZFC is consistent then so is ZFC + `there are no
P-points'.
(See~\cite{Sh:PIF}.)

We could not resolve the following question.

\begin{question}
If $\cV$ is an ultrafilter such that $F_{\cV} (\BH) \neq \bbC$,
does it follow that $\cV$ is flat?
\end{question}

As pointed out in the introduction,
tools from the logic of metric structures~(\cite{BYBHU}) are very
relevant to the study of ultrapowers of C*-algebras.
(See~\cite{FaHaSh:Model} for recent applications.)
For example, it
would be interesting to reformulate some of the results
of~\cite{Kirc:Central} using the language of model theory. In
particular, can the notion of $\sigma$-sub-Stonean
(\cite[Definition~1.4]{Kirc:Central}) be replaced with the notion of
$\aleph_1$-saturated (\cite[Definition~7.5]{BYBHU}, the case when
$\kappa=\aleph_1$, the least uncountable cardinal)?


\begin{thebibliography}{33}

\bibitem{Bla:Cardinal}
A.~Blass,
{\emph{Combinatorial cardinal characteristics of the continuum}},
to appear in {\emph{Handbook of Set Theory}},
M.~Foreman, M.~Magidor, and A.~Kanamori, editors.
Available at http://www.math.lsa.umich.edu/$\sim$ablass/set.html

\bibitem{BYBHU}
I.~Ben~Yaacov, A.~Berenstein, C.W.\  Henson, and A.~Usvyatsov,
{\emph{Model theory for metric structures}},
in {\emph{Model Theory with
   Applications to Algebra and Analysis, Vol. II}}
(Z.~Chatzidakis et~al., eds.),
Lecture Notes series of the London Math.\  Society, no.~350,
Cambridge University Press, 2008, pp.~315--427.

\bibitem{Cn2} A.~Connes,
{\emph{Outer conjugacy classes of automorphisms of factors}}, Ann.\  %
Sci.\  \'{E}c.\  Norm.\  Sup.\  S\'{e}r.~4, {\textbf{8}} (1975),
383--420.

\bibitem{Dx} J.~Dixmier, {\emph{Von Neumann Algebras}}, North-Holland,
Amsterdam, New York, Oxford, 1981.

\bibitem{Fa:Calkin}
I.~Farah, {\emph{All automorphisms of the {C}alkin algebra are
inner}}, preprint (arXiv: 0705.3085v7 [math.OA]).

\bibitem{Fa:Semiselective}
I.~Farah,
{\emph{Semiselective coideals}},
Mathematika {\textbf{45}} (1998), 79--103.

\bibitem{Fa:Relative}
I.~Farah,
{\emph{The relative commutant of separable {C*}-algebras
of real rank zero}},
J.~Funct.\  Anal.\  {\textbf{256}} (2009), 3841--3846.

\bibitem{FaHaSh:Model}
I.~Farah, B.~Hart, and D.~Sherman,
{\emph{Model theory of operator algebras I: Stability}},
preprint (arXiv: 0908.2790v1 [math.OA]).

\bibitem{FaHaSh:Model2}
I.~Farah, B.~Hart, and D.~Sherman,
{\emph{Model theory of operator algebras II: Model theory}},
in preparation, 2009.

\bibitem{FaSt:Flat}
I.~Farah and J.~Stepr\={a}ns, {\emph{Flat ultrafilters}}, in
preparation, 2009.

\bibitem{GeHa} L.~Ge and D.~Hadwin,
{\emph{Ultraproducts of {C*}-algebras}}, Recent advances in
operator theory and related topics (Szeged, 1999), Oper.\  Theory
Adv.\  Appl., vol.~127, Birkh\"{a}user, Basel, 2001, pp.~305--326.

\bibitem{Iz} M.~Izumi,
{\emph{Finite group actions on C*-algebras with the
   Rohlin property.~I}},
Duke Math.~J.\  {\textbf{122}} (2004), 233--280.

\bibitem{Jn} V.~F.\  R.\  Jones,
{\emph{Actions of finite groups on the hyperfinite type II$_1$
factor}}, Mem.\  Amer.\  Math.\  Soc.\  {\textbf{28}} (1980), no.~237.

\bibitem{Kr} E.~Kirchberg, {\emph{The classification of purely
infinite C*-algebras using Kasparov's theory}}, preliminary preprint
(3rd draft).

\bibitem{Kirc:Central}
E.~Kirchberg, {\emph{Central sequences in C*-algebras and
strongly purely infinite algebras}}, Operator Algebras: The Abel
Symposium 2004, Abel Symp., vol.~1, Springer, Berlin, 2006,
pp.~175--231.

\bibitem{KP1} E.~Kirchberg and N.~C.\  Phillips,
{\emph{Embedding of exact C*-algebras in the Cuntz algebra
   ${\mathcal{O}}_2$}},
J.~reine angew.\  Math.\  {\textbf{525}} (2000), 17--53.

\bibitem{Ku:Some}
K.~Kunen, {\emph{Some points in {$\beta N$}}},
Math.\  Proc.\  Cambridge Philos.\  Soc.\  %
{\textbf{80}} (1976), no.~3, 385--398.

\bibitem{Ku:Book}
K.~Kunen,
{\emph{Set Theory: An Introduction to Independence Proofs}},
North--Holland, Amsterdam, 1980.

\bibitem{Mat:Happy}
A.~R.\  D.\  Mathias, {\emph{Happy families}}, Annals of Mathematical
Logic {\textbf{12}} (1977), 59--111.

\bibitem{McDuff:Central}
D.~McDuff, {\emph{Central sequences and the hyperfinite factor}},
Proc.\  London Math.\  Soc.\  {\textbf{21}} (1970), 443--461.

\bibitem{Oc} A.~Ocneanu,
{\emph{Actions of Discrete Amenable Groups on von Neumann Algebras}},
Springer-Verlag Lecture Notes in Math.\  no.~1138, Springer-Verlag,
Berlin, 1985.

\bibitem{Ph} N.~C.\  Phillips, {\emph{A classification theorem for
nuclear purely infinite simple C*-algebras}}, Documenta Math.\  %
{\textbf{5}} (2000), 49--114 (electronic).

\bibitem{PhWe:Calkin}
N.~C.\  Phillips and N.~Weaver,
{\emph{The {C}alkin algebra has outer automorphisms}},
Duke Math.\  J.\  {\textbf{139}} (2007), 185--202.

\bibitem{Sh:PIF}
S.~Shelah,
{\emph{Proper and Improper Forcing}},
2nd.\  ed.,
Perspectives in Mathematical Logic,
Springer-Verlag, Berlin, 1998.

\bibitem{She:Divisible}
D.~Sherman,
{\emph{Divisible operators in von Neumann algebras}},
Illinois J.\  Math., to appear.

\bibitem{Tk3}  M.~Takesaki, {\emph{Theory of Operator Algebras III}},
Springer-Verlag, Berlin, etc., 2003.

\end{thebibliography}
\end{document}